\documentclass{article}

\usepackage[left=1in,right=1in,top=1in,bottom=1in]{geometry}

\usepackage{graphicx}
\usepackage{amssymb}
\usepackage{amsmath}
\usepackage{verbatim,booktabs}
\usepackage{subfigure,epsfig,url,psfrag}
\usepackage{float}
\usepackage{mathrsfs, enumitem}

 \usepackage[usenames,dvipsnames]{pstricks}
 \usepackage{epsfig}
 \usepackage{pst-grad} 
 \usepackage{pst-plot} 
\usepackage{soul}

\makeatletter
\newcommand{\setwindow}[5]{
\def\xmin{#1}%
\def\ymin{#2}%
\def\xmax{#3}%
\def\ymax{#4}%
\pstFPsub\viewingwidth{#3}{#1}%
\pstFPdiv\result{\strip@pt#5}{\viewingwidth}%
\psset{unit=\result pt}}
\makeatother

\usepackage{amsthm}

\newtheorem{proposition}{Proposition}
\newtheorem{lemma}{Lemma}

\newtheorem{remark}{Remark}

\def\R{{\mathbb R}}
\def\conv{{\rm conv}}
\def\conc{{\rm conc}}

\def\ie{{i.e.,} }

\def\S{{\mathcal S}}
\def\H{{\mathcal H}}
\def\C{{\mathcal C}}
\def\L{{\mathcal L}}
\def\I{{\mathcal I}}

\def\X{{\mathcal X}}

\def\K{{\mathcal K}}
\def\I{{\mathcal I}}
\def\T{{\mathcal T}}
\def\D{{\mathcal D}}

\def\P{{\mathcal P}}
\def\Q{{\mathcal Q}}

\def\01{\ensuremath{0\mathord{-}1}}

\allowdisplaybreaks

\newcommand{\bi}{\begin{list}{$\bullet$}{\setlength{\parsep}{0pt}\setlength{\itemsep}{0pt}}}

\newcounter{claim} 

\title{The circle packing problem: a theoretical comparison of various convexification techniques
\thanks{The author was partially funded by AFOSR grant FA9550-23-1-0123.}}

\author{Aida Khajavirad
\thanks{Department of Industrial and Systems Engineering,
             Lehigh University.
             E-mail: {\tt aida@lehigh.edu}.}
             }

\begin{document}

\maketitle

\begin{abstract}
We consider the problem of packing congruent circles with the maximum radius in a unit square as a mathematical optimization problem. Due to the presence of non-overlapping constraints,
this problem is a notoriously difficult nonconvex quadratically constrained optimization
problem, which possesses many local optima. We
consider several popular convexification techniques, giving rise to linear programming relaxations and semidefinite programming relaxations for the circle packing problem.
We compare the strength of these relaxations theoretically,
thereby proving the conjectures by Anstreicher~\cite{kurt09}.
Our results serve as a theoretical justification for the ineffectiveness of existing machinery for convexification of non-overlapping constraints.
\end{abstract}


\section{Introduction}
The problem of finding the maximum radius $r$ of $n$ identical non-overlapping circles that fit in a unit square is a classic problem in discrete geometry.
It is well-known that this problem can be equivalently stated as:

\medskip
\emph{Locate $n$ points in a unit square, such that the minimum distance between any two points is maximal.}
\medskip\\
Denote by $(x_i, y_i)$, $i \in [n]:=\{1, \ldots, n\}$ the coordinate of the $i$th point to be located in the unit square. It then follows
that the above problem can be stated as the following optimization problem:
\begin{align}\label{problemCP}
\tag{CP}
{\rm max}   \quad & \gamma\nonumber\\
 {\rm s.t.} \quad   & (x_j-x_i)^2 + (y_j - y_i)^2 \geq \gamma,\quad 1 \leq i < j \leq n,\label{noverlap} \\
       &  x \in [0, 1]^n, \quad y \in [0, 1]^n, \nonumber
\end{align}
where $\gamma$ denotes the minimum squared pair-wise distance of the points in the unit square. The
radius $r$ of the $n$ circles that can be packed into the unit square is then given by $r = \frac{\sqrt{\gamma}}{2(1+\sqrt{\gamma})}$.
Throughout this paper, we refer to Problem~\eqref{problemCP} as the \emph{circle packing problem}. In spite of its simple formulation, the circle
packing problem is a difficult nonconvex optimization problem with a large number of locally optimal solutions.
This nonconvexity is due to the presence of \emph{non-overlapping constraints} defined by inequalities~\eqref{noverlap}. In fact, non-overlapping constraints appear in a variety
of applications  including circular cutting,
communication networks and facility layout  problems (see~\cite{cast08} for a detailed review of industrial applications).

The circle packing problem and its variants have been studied extensively by the optimization community~(see~\cite{szabo05} for a review).
Yet, we are unable to solve Problem~\eqref{problemCP} to global optimality for $n > 10$ with any of the state-of-the-art general-purpose mixed-integer nonlinear programming (MINLP) solvers~\cite{IdaNick18,mf14,VigGle16}
within a few hours of CPU-time.
This surprisingly poor performance is mainly due to the generation of weak upper bounds on the optimal value of $\gamma$, which in turn is caused by our inability to
effectively convexify a collection of non-overlapping constraints.

In this paper, we perform a theoretical study of existing techniques to convexify a nonconvex set defined by a collection of non-overlapping constraints.
We focus on the circle packing problem as a prototypical example of these optimization problems because its simple formulation allows us to solve the convex relaxations analytically and conduct a theoretical assessment of their relative strength.
%
In Section~\ref{sec: TW}, we consider several linear programming (LP) relaxations of Problem~\eqref{problemCP}
obtained by replacing each non-overlapping constraint with its convex hull over the domain of variables.
We refer to such relaxations as \emph{single-row} LP relaxations of the circle packing problem.
We prove that these relaxations are quite weak, confirming their ineffectiveness when embedded in
MINLP solvers. In Section~\ref{sec: MT}, we propose stronger LP relaxations of Problem~\eqref{problemCP} by convexifying multiple
non-overlapping constraints simultaneously. Namely, we consider a reformulation of the circle packing problem whose relaxation is closely
related to the Boolean quadric polytope (BQP)~\cite{Pad89}, a well-studied polytope in discrete optimization.
By building upon existing results on the facial structure of the BQP, we present
\emph{multi-row} LP relaxations of the circle packing problem. While our multi-row relaxations are considerably stronger than the single-row relaxations, their optimality gap grows quickly as $n$ increases, making them also ineffective in practice.
In Section~\ref{sec: sdp}, we examine the strength of SDP relaxations for the circle packing problem. We prove that the upper bounds achieved
by these relaxations are identical or worse than the bounds obtained by the proposed LP relaxations. In essence, our results provide a theoretical justification for the ineffectiveness of existing machinery for convexification of non-overlapping constraints.


\section{Single-row LP relaxations}
\label{sec: TW}

\paragraph{The basic approach.} Perhaps the most intuitive approach to obtain an LP relaxation for the circle packing problem is to
replace the nonconvex set defined by a \emph{single} non-overlapping constraint by its convex hull.
Denote by $\conv(\C)$ the convex hull of a set $\C$ and denote $\conc_{\X} f$ the concave envelope of the function $f$ over the convex set $\X$.
Then the concave envelope
of $f(x)=(x_1-x_2)^2$ over the box $\H = [0,u_1]\times [0, u_2]$ is given by:
\begin{equation}\label{concv}
\conc_{\H} f(x) = \min\big\{u_1 x_1+u_2 x_2, \; 2u_1u_2+(u_1-2u_2)x_1+(u_2-2u_1)x_2\big\}.
\end{equation}
Since non-overlapping constraints are separable in $x$ and $y$, for each $1\leq i < j \leq n$, we have
\begin{align*}
\conv &\big\{(x_i, x_j, y_i, y_j, \gamma): (x_j-x_i)^2 + (y_j - y_i)^2 \geq \gamma, \; x \in [0, 1]^2, y \in [0, 1]^2\big\} =\\
& \big\{(x_i, x_j, y_i, y_j, \gamma): \conc_{[0,1]^2} f(x) + \conc_{[0,1]^2} f(y) \geq \gamma, \;x \in [0, 1]^2, y \in [0, 1]^2\big\}.
\end{align*}
Therefore, the following LP provides an upper bound on the optimal value of Problem~\eqref{problemCP}:
\begin{align*}\label{problemTW}\tag{TW}
{\rm max}   \quad & \gamma\\
 {\rm s.t.} \quad  & \left.\begin{array}{ll}
 x_i+x_j + y_i+y_j \geq \gamma\\
	-x_i-x_j + y_i+y_j+2 \geq \gamma \\
      x_i+x_j - y_i-y_j+2 \geq \gamma\\
     -x_i-x_j - y_i-y_j+4 \geq \gamma
     \end{array}\right\},
	\quad 1 \leq i < j \leq n \\
       &  x \in [0,1]^n, y \in [0,1]^n.
\end{align*}

\begin{remark}\label{re1}
Define $X_{ij} := x_i x_j$  for all $1 \leq i \leq j \leq n$. We replace the set
$\{(x_i, x_j, X_{ij}): X_{ij} \geq x_i x_j, x \in [0,1]^2\}$ by its convex hull for all $1 \leq i < j \leq n$.
Similarly, we replace the set $\{(x_i, X_{ii}): X_{ii} \leq x^2_i, x \in [0,1]\}$ by its convex hull for all $i \in [n]$. Symmetrically, we utilize the same arguments for $y$ variables,
to obtain the following relaxation of the feasible region of Problem~\eqref{problemCP} in a lifted space:
\begin{align*}
\S = \Big\{ (x, y, X, Y, \gamma): \; & X_{ii}-2 X_{ij} + X_{jj} + Y_{ii}-2 Y_{ij} + Y_{jj} \geq \gamma, \;
X_{ij} \geq 0, \; X_{ij} \geq x_i + x_j-1, \; Y_{ij} \geq 0, \\
&   Y_{ij} \geq y_i + y_j-1, \; \forall 1 \leq i < j \leq n, \; X_{ii} \leq x_i, \;  Y_{ii} \leq y_i, \; \forall i \in [n] \Big\}.
\end{align*}
It can be checked that the projection of $\S$ onto the $(x, y, \gamma)$ space is given by the constraint set of Problem~\eqref{problemTW}. This lifted relaxation is often referred to as
the \emph{first-level Reformulation Linearization Technique (RLT) relaxation}~\cite{sa99} of the circle packing problem
and is utilized by most of the general-purpose MINLP solvers to find upper bounds for this problem.
\end{remark}

The circle packing problem is highly symmetric and utilizing~\emph{symmetry-breaking constraints}
is beneficial for solving this problem to global optimality~\cite{kurt09,costa13}.
In the following, we present sharper LP relaxations for Problem~\eqref{problemCP} by utilizing a couple of simple symmetry-breaking type constraints.

\vspace{-0.1in}
\paragraph{Tighter variable bounds.}
Let $n_x = \lceil \frac{n}{2} \rceil$ and $n_y = \lceil \frac{n_x}{2} \rceil$.
By symmetry, we can assume that at any optimal solution of Problem~\eqref{problemCP}, we have:
\begin{equation}\label{bnds}
0\leq x_i \leq \frac{1}{2}, \quad i =1,\ldots, n_x, \qquad 0\leq y_i \leq \frac{1}{2}, \quad  i =1,\ldots, n_y.
\end{equation}
Utilizing these bounds for $x$ and $y$ and
replacing each bivariate quadratic function by its concave envelope over the corresponding box, given by~\eqref{concv}, we obtain an LP relaxation, which we refer to as Problem~(TWbnd).
%
%


\vspace{-0.1in}
\paragraph{Order constraints.}
We can impose an order on $x$ variables by adding the inequalities
\begin{equation}\label{order}
0 \leq x_1 \leq x_2 \leq \cdots \leq x_n \leq 1,
\end{equation}
to Problem~\eqref{problemCP}. Subsequently, for each $1 \leq i < j \leq n$, we replace $(x_j-x_i)^2$, by its concave envelope
over the triangular region $0 \leq x_i \leq x_j \leq 1$, and we replace $(y_j-y_i)^2$, by its concave envelope over $y_i,y_j \in [0,1]$ to obtain the following LP relaxation of~\eqref{problemCP}:
\begin{align*}\label{problemTWord}\tag{TWord}
{\rm max}   \quad & \gamma\\
 {\rm s.t.} \; & \left.\begin{array}{ll}
	 x_j-x_i + y_i+y_j \geq \gamma\\
     x_j-x_i - y_i-y_j+2 \geq \gamma
	\end{array}\right\},\quad 1 \leq i < j \leq n \\
       &  0 \leq x_1 \leq \cdots \leq x_n \leq 1, \quad y \in [0,1]^n.
\end{align*}

\begin{remark}\label{re2}
Consider the first-level RLT relaxation of the feasible region of Problem~\eqref{problemCP} defined in Remark~\ref{re1}.
We can strengthen this relaxation by utilizing order constraints~\eqref{order} to generate the RLT inequalities:
\begin{eqnarray}\label{RLTord}
X_{ii} \leq X_{ij}, \quad x_i-X_{ij} \leq x_j - X_{jj}, \quad 1 \leq i < j \leq n.
\end{eqnarray}
Suppose that we add inequalities~\eqref{RLTord} to the set $\S$ defined in Remark~\ref{re1}.
Then it can be shown that the projection of this new set
onto the $(x,y,\gamma)$ space coincides with the feasible region of Problem~\eqref{problemTWord}.
\end{remark}

\vspace{-0.2in}
\paragraph{Best single-row LP relaxations.} Combining the two symmetry-breaking constraints described above leads to stronger relaxations of Problem~\eqref{problemCP}.
Consider the tigher bounds on $x$ and $y$ as given by~\eqref{bnds}. By imposing these bounds, the order constraint $x_{n_y} \leq x_{n_y+1}$ is no longer valid.
Thus, we impose the order constraints as follows:
\begin{equation}\label{semiord}
x_i \leq x_{i+1}, \quad \forall i \in [n-1] \setminus \{n_y\}.
\end{equation}
Next, we replace each quadratic term by its concave envelope over the corresponding rectangular, triangular, or trapezoidal domain to obtain an LP relaxation, which we refer to as Problem~(TWcomb).

\medskip

The next proposition provides optimal values of LP relaxations defined above. Anstreicher~\cite{kurt09} conjectured these bounds and verified them
numerically for $3 \leq n \leq 50$ (see Conjectures~4 and~5 in~\cite{kurt09}).

\begin{proposition}\label{th1}
Consider the single-row LP relaxations of the circle packing problem defined above:
\begin{itemize}
\item [(i)] The optimal value of Problem~\eqref{problemTW} is $\gamma^* = 2$ for all $n \geq 2$.

\item [(ii)] The optimal
value of Problem~(TWbnd) is $\gamma^* =\frac{1}{2}$ for all $n \geq 5$.

\item [(iii)] The optimal value of Problem~(TWord) is $\gamma^* = 1+\frac{1}{n-1}$ for all $n\geq 2$.

\item [(iv)] The the optimal value of Problem~(TWcomb) is $\gamma^* =\frac{1}{4}(1+\frac{1}{{\lfloor(n-1)/4\rfloor}})$ for all $n \geq 5$.
\end{itemize}
\end{proposition}

\begin{proof}
To find the optimal value of each LP, we first find an upper bound on its objective function value by considering a specific subset of constraints and subsequently show that the upper bound is sharp by providing a feasible point that attains the same objective value.

\vspace{0.1in}
\noindent
\textbf{Part~$(i)$.} Consider the four inequality constraints of Problem~\eqref{problemTW} for some $1 \leq i < j \leq n$. Summing up these inequalities we obtain $\gamma \leq 2 := \tilde \gamma$.
Now consider the point $\tilde x_i = \tilde y_i = \frac{1}{2}$, for $i \in [n]$. Substituting $(\tilde x, \tilde y, \tilde \gamma)$ in the constraints of~\eqref{problemTW} we get
$\frac{1}{2}+ \frac{1}{2} + \frac{1}{2} + \frac{1}{2} \geq 2$, $-\frac{1}{2}-\frac{1}{2}+\frac{1}{2}+\frac{1}{2}+2 \geq 2$, $\frac{1}{2}+\frac{1}{2}-\frac{1}{2}-\frac{1}{2}+2\geq 2$, and $-\frac{1}{2}-\frac{1}{2}-\frac{1}{2}-\frac{1}{2}+4 \geq 2$. Hence, $(\tilde x, \tilde y, \tilde \gamma)$  is feasible for~\eqref{problemTW}, implying that $\gamma^*=\tilde \gamma = 2$ for all $n \geq 2$.

\vspace{0.1in}
\noindent
\textbf{Part~$(ii)$.}
Suppose that $n \geq 5$, so that $n_y \geq 2$; \ie by~\eqref{concv} there exists a pair $1 \leq i < j \leq n_y$ satisfying
\begin{equation}\label{sm}
x_i+x_j+y_i+y_j \geq 2\gamma, \; x_i+x_j-y_i-y_j+1 \geq 2\gamma, \; -x_i-x_j+y_i+y_j+1 \geq 2 \gamma, \; -x_i-x_j-y_i-y_j+2 \geq 2 \gamma.
\end{equation}
Summing up inequalities~\eqref{sm} we obtain $\gamma \leq \frac{1}{2} := \tilde \gamma$.
Consider the point $\tilde x_i = \tilde y_i = \frac{1}{4}$ for
$i \in [n_y]$ and $\tilde x_i = \tilde y_i = \frac{1}{2}$ for $n_y < i \leq n$. We claim that
$(\tilde x, \tilde y, \tilde \gamma)$ is a feasible point of~(TWbnd), implying that $\gamma^* = \frac{1}{2}$ for all $n \geq 5$. To see this, first note that this point satisfies inequalities~\eqref{sm} for $1 \leq i < j \leq n_y$. Denote by $\D, \D'$ two convex sets such that $\D' \subset \D$ and consider a function $f(z)$ defined over these sets.
Then for any $z \in \D'$ we have $\conc_{\D'} f(z) \leq \conc_{\D} f(z)$. Hence to show feasibility of $(\tilde x, \tilde y, \tilde \gamma)$, two cases remain to consider:

\begin{itemize}[leftmargin=*]
    \item for each $1 \leq i \leq n_y < j \leq n_x$,
    by~\eqref{concv}, Problem~(TWbnd) contains the following inequalities:
    $x_i+x_j+y_i+2 y_j \geq 2 \gamma, \;
	x_i+x_j-3 y_i+2 \geq 2 \gamma, \;
	-x_i-x_j+y_i+2 y_j +1 \geq 2 \gamma, \;
	-x_i-x_j-3 y_i+3 \geq 2 \gamma$. Substituting
    $(x_i,y_i)=(\frac{1}{4},\frac{1}{4})$,
     $(x_j,y_j)=(\frac{1}{2},\frac{1}{2})$, $\gamma=\tilde \gamma$ in these inequalities yields:
    $\frac{1}{4}+\frac{1}{2}+\frac{1}{4}+1 \geq 1,\;
	\frac{1}{4}+\frac{1}{2}-\frac{3}{4}+2 \geq 1,\;
	-\frac{1}{4}-\frac{1}{2}+\frac{1}{4}+1 +1 \geq 1,\;
	-\frac{1}{4}-\frac{1}{2}-\frac{3}{4}+3 \geq 1$.

    \item for each $n_y \leq i  < j \leq n_x$, by~\eqref{concv}, Problem~(TWbnd) contains the following inequalities: $x_i+x_j +2 y_i+2 y_j \geq 2 \gamma, \;
	x_i+x_j-2 y_i-2 y_j+4 \geq 2 \gamma, \;
	-x_i-x_j+2 y_i+2 y_j +1 \geq 2 \gamma, \;
	-x_i-x_j-2 y_i-2 y_j+5 \geq 2 \gamma$.
     Substituting
    $(x_i,y_i)=(\frac{1}{2},\frac{1}{2})$,
     $(x_j,y_j)=(\frac{1}{2},\frac{1}{2})$,
     $\gamma=\tilde \gamma$ in these inequalities yields:
     $\frac{1}{2}+\frac{1}{2} +1+1 \geq 1, \;
	\frac{1}{2}+\frac{1}{2}-2-2+4 \geq 1, \;
	-\frac{1}{2}-\frac{1}{2}+2+2 +1 \geq 1, \;
	-\frac{1}{2}-\frac{1}{2}-2-2+5 \geq 1$.
\end{itemize}
Hence $(\tilde x, \tilde y, \tilde \gamma)$ is a feasible point of~(TWbnd), implying that $\gamma^* = \frac{1}{2}$ for all $n \geq 5$.

\vspace{0.1in}
\noindent
\textbf{Part~$(iii)$.}  Consider a pair of inequality constraints of Problem~\eqref{problemTWord} for some $1\leq i <j \leq n$.
Summing these inequalities we obtain $\gamma \leq x_j - x_i +1$. Consider a subset of such inequalities with $j=i+1$. Then the optimal value of the following problem is an upper bound on the optimal value of Problem~\eqref{problemTWord}:
\begin{align*}
{\rm max}   \quad & \gamma \nonumber\\
 {\rm s.t.} \quad  & x_{i+1}-x_i +1\geq \gamma,\quad i\in [n] \nonumber\\
    & 0 \leq x_1 \leq \ldots \leq x_n \leq 1.\nonumber
\end{align*}
The optimal value of the above problem is attained at $\tilde x_i = \frac{i-1}{n-1}$, $i\in [n]$, and $\tilde \gamma = 1+\frac{1}{n-1}$. Therefore $\tilde \gamma$ is an upper bound on the optimal
value of Problem~\eqref{problemTWord}.
Now consider the point $\tilde x_i = \frac{i-1}{n-1}$, and $\tilde y_i = \frac{1} {2}$ for all $i\in [n]$. First note that this point satisfies order constraints on $x$ and bound constraints on $y$.
Substituting $(\tilde x, \tilde y, \tilde\gamma)$ in the remaining constraints of Problem~\eqref{problemTWord}, we obtain
$1-\frac{j-i-1}{n-1} \leq \tilde y_i + \tilde y_j \leq 1+\frac{j-i-1}{n-1}$ for all $1 \leq i < j \leq n$, which is satisfied since $\tilde y_i + \tilde y_j= 1$.
Therefore  $\gamma^* = \tilde \gamma$.

\vspace{0.1in}
\noindent
\textbf{Part~$(iv)$.} Suppose $n \geq 5$ so that $n_y \geq 2$.
Summing each pair of inequality constraints of Problem~(TWcomb) for a given $i,j$ with $1 \leq i < j \leq n_y$ we obtain $\gamma \leq x_j - x_i +\frac{1}{2}$.
Using a similar line of arguments as in Part~$(iii)$ it follows that and upper bound on the optimal value of Problem~(TWcomb) is given by
$\tilde\gamma = \frac{1}{4}(1+\frac{1}{n_y-1})$.
Consider the point $\tilde x_i = \frac{i-1}{2(n_y-1)}$, $\tilde y_i = \frac{1}{4}$ for $i \in [n_y]$,
and $\tilde x_i = \tilde y_i = \frac{1}{2}$ for $i = n_y+1,\ldots,n$. Using a similar line of basic arguments as in Part~$(ii)$ and Part~$(iii)$, one can show feasibility of
$(\tilde x, \tilde y, \tilde \gamma)$ for Problem~(TWcomb), implying $\gamma^* = \tilde \gamma$.
\end{proof}

\vspace{-0.2in}
\section{Multi-row LP relaxations}
\label{sec: MT}

In this section, we present stronger LP relaxations of the circle packing problem
by convexifying multiple non-overlapping constraints simultaneously.
We start by introducing a reformulation of Problem~\eqref{problemCP} which we will use for later developments:
\begin{align*}\label{problemCPr}\tag{CPr}
{\rm max}   \quad & \gamma\\
 {\rm s.t.}  \quad  & (x_j-x_i)^2 + (y_j - y_i)^2 \geq \beta_{ij}, \quad 1 \leq i < j \leq n, \\
       &  \beta_{ij} \geq \gamma, \quad 1 \leq i < j \leq n,\\
      &  x \in [0, 1]^n, \quad y \in [0, 1]^n.
\end{align*}
Define the sets
\begin{align}
& \P := \{(x,y, \beta, \gamma): \; (x_j-x_i)^2 + (y_j - y_i)^2 \geq \beta_{ij}, \; \forall 1 \leq i < j \leq n, \; x, y \in [0, 1]^n\}\label{rev1}\\
&\K := \{(x,y, \beta, \gamma): \beta_{ij} \geq \gamma, \; \forall \; 1 \leq i < j \leq n, \; x, y \in \R^n\}.\label{rev2}
\end{align}
Since $\conv(\P \cap \K) \subseteq \conv(\P) \cap \conv(\K)$ and $\K$ is a convex cone, we deduce that
$\conv(\P) \cap \K$ is a convex relaxation of the feasible region of Problem~\eqref{problemCPr}.
In the following, we obtain an extended formulation for the convex hull of $\P$.
To this end, let us formally define the BQP first introduced by Padberg~\cite{Pad89}:
$$
{\rm BQP}_n := \conv\big\{(x,X): X_{ij} = x_i x_j, \; \forall 1 \leq i < j \leq n, \; x \in \{0, 1\}^n\big\}.
$$
Denote by $\S_x$ the set described by all facet-defining inequalities of BQP of the form $a x + b X \leq c$ with $b_{ij} \leq 0$, for all $1 \leq i < j \leq n$. Similarly, define $\S_y$ by replacing $(x,X)$ with $(y,Y)$ in the description of $\S_x$.

\begin{lemma}\label{lem1}
Consider the set $\P$ defined by~\eqref{rev1}. Then an extended formulation for $\conv(\P)$ is given by:
\begin{align*}
    &  X_{ii} -2 X_{ij} + X_{jj} + Y_{ii} -2 Y_{ij} + Y_{jj}\geq \beta_{ij}, \; 1 \leq i < j \leq n\\
    & (x,X) \in \S_x, \; (y,Y) \in \S_y, \; X_{ii} \leq x_i, Y_{ii} \leq y_i, \; i \in [n].
\end{align*}
\end{lemma}

\begin{proof}
Consider the set
\begin{align*}
\Q := \big\{(x,y, X, Y): x^2_i \geq X_{ii}, \; y^2_i \geq Y_{ii}, \forall i \in [n],
                        X_{ij} \geq x_i x_j, \; Y_{ij} \geq y_i y_j, \forall 1 \leq i < j \leq n, \; x, y \in [0, 1]^n \big\}.
\end{align*}
Denote by $\bar \P$ the projection of $\P$ onto the space $(x, y, \beta)$. Define the linear mapping $\L: (x, y, X, Y) \rightarrow (x, y, \beta)$, where
$\beta_{ij} = X_{ii} - 2 X_{ij} + X_{jj} + Y_{ii} - 2 Y_{ij} + Y_{jj}$ for all $1 \leq i < j \leq 1$.
To characterize $\conv(\P)$, it suffices to characterize $\conv(\Q)$, as $\bar \P$ is the image of $\Q$ under the linear mapping $\L$, \ie
$\conv (\bar \P) = \conv (\L \Q) = \L \conv(\Q)$.
Now consider the set $\Q$; denote by $\bar \Q_x$ (resp. $\bar \Q_y$), the set obtained by dropping the inequalities containing $(y, Y)$ variables (resp. $(x, Y)$ variables)
from the description of $\Q$. It then follows that $\conv (\Q) = \conv (\bar \Q_x) \cap \conv (\bar \Q_y)$.
Let $\Q_x$ denote the projection of $\bar \Q_x$ onto the $(x, X)$ space and define
$$
\Q^b_x := \{(x, X): X_{ij} \geq x_i x_j, \; \forall 1 \leq i < j \leq n, \; x \in [0, 1]^n\}, \; \Q^s_x := \{(x, X): X_{ii} \leq x^2_i, \; \forall i \in [n], \; x \in [0,1]^n\}.
$$
Clearly, $\conv (\Q^s_x) = \{(x, X): X_{ii} \leq x_i, \; \forall i \in [n], \; x \in [0,1]^n\}$ and $\conv (\Q_x) \subseteq \conv (\Q^b_x) \cap \conv (\Q^s_x)$.
Moreover, $\conv (\Q_x) \supseteq \conv (\Q^b_x) \cap \conv (\Q^s_x)$, as the projection of each point in the convex hull of $\{(x_i, X_{ii}) : X_{ii} \leq x^2_i, \; x_i \in [0,1]\}$
onto the $x_i$ space can be uniquely written as a convex combination of the two end points of $[0, 1]$. It then follows that $\conv (\Q_x) = \conv (\Q^b_x) \cap \conv (\Q^s_x)$.

Now consider $\Q^b_x$. It can be shown that $\conv(\Q^b_x)$ is polyhedral. Moreover, the projection of the vertices
of $\conv(\Q^b_x)$ onto the space of $x$ variables coincides with the vertices of the unit hypercube.
To see this, let $\hat x$ denote a binary vector in $\R^n$ and let $\hat X_{ij} = \hat x_i \hat x_j$ for all $1 \leq i < j \leq n$.
To show that $(\hat x, \hat X)$ is a vertex of $\conv(\Q^b_x)$, it suffice to characterize a linear function $c^T X + d^T x$ whose maximum over $\conv(\Q^b_x)$
is uniquely attained at $(\hat x, \hat X)$. Let $k = \sum_{i=1}^n {\hat x_j}$. Define $c_{ij} = -1$ for all $1 \leq i < j \leq n$, $d_j = -1$ for all $j$ such that $\hat x_j = 0$
and $d_j = k$ for all $j$ such that $\hat x_j = 1$. Using the fact that the function $f(m) = k m - m(m-1)/2$, where $0 \leq m \leq k$ is strictly increasing, we conclude that
the unique maximizer of $c^T X + d^T x$ over $\conv(\Q^b_x)$ is given by $(\hat x, \hat X)$.
It then follows that $\conv(\Q^b_x)$ is a relaxation of the BQP. In fact, the facets of $\conv(\Q^b_x)$ are precisely
those of the BQP of the form $a x + b X \leq c$ with $b_{ij} \leq 0$, for all $1 \leq i < j \leq n$. To see this, first note that
$\conv(\Q^b_x)$ and BQP have the same set of vertices; however, while BQP is bounded, $\conv(\Q^b_x)$ is an unbounded
polyhedron whose recession cone is given by $\{(x, X): X_{ij} \geq 0, \; \forall 1 \leq i < j \leq n \}$.
It then follows that a facet-defining inequality for BQP defines a facet of $\conv(\Q^b_x)$ if and only if it is valid for
$\conv(\Q^b_x)$. To see this, suppose that $a x + b X \leq c$ defines a facet of BQP and $b_{ij} > 0$ for some $1 \leq i < j \leq n$. Consider
a point $(\tilde x, \tilde X) \in {\rm BQP}_n$.
We can construct a point in $\conv(\Q^b_x)$ by making the value of $\tilde X_{ij}$ arbitrarily large while keeping all other components unchanged,
so as to violate $a x + b X \leq c$.
Symmetrically,
we obtain a characterization of $\conv(\Q_y)$.
\end{proof}
Hence, replacing $\P \cap \K$ with $\conv(\P) \cap \K$ in Problem~\eqref{problemCPr}, and using Lemma~\ref{lem1}, we obtain a new LP relaxation for Problem~\eqref{problemCP}:
\begin{align}\label{problemMT}\tag{MT}
{\rm max}   \quad & \gamma\nonumber\\
 {\rm s.t.}  \quad  & X_{ii} -2 X_{ij} + X_{jj} + Y_{ii} -2 Y_{ij} + Y_{jj}\geq \gamma, \; 1 \leq i < j \leq n\nonumber\\
       & (x, X) \in \S_x, \; (y, Y) \in \S_y\nonumber \\
       &   X_{ii} \leq x_i, \; Y_{ii} \leq y_i, \; i \in [n]\nonumber.
\end{align}
It is well-understood that BQP has a very complex structure. In fact, an explicit description for BQP is
only available for $n \leq 6$~\cite{dl:97}, implying that a closed-form description of $\S_x$ for $n > 6$ is not available either. Next we present a relaxation of
$\S_x$; denote by $\C_x$ the polyhedron defined by all so-called \emph{clique inequalities}:
\begin{equation}\label{cliqueP}
\alpha \sum_{i=1}^m{x_i} - \sum_{1 \leq i < j \leq m} {X_{ij}} \leq \frac{\alpha (\alpha+1)}{2}, \quad \forall m \in \{3,\ldots, n\}, \; \forall 1 \leq \alpha \leq \max\{m-2,1\}.
\end{equation}
Clique inequalities induce facets of the BQP~\cite{Pad89}, and hence are facet-defining for
$\S_x$ as well.
We should remark that $\S_x = \C_x$ for $n \leq 4$ while $\S_x \subset \C_x$ for $n \geq 5$ (see Chapter 29 of~\cite{dl:97}). Replacing $\S_x$ by $\C_x$ in Problem~\eqref{problemMT},
we obtain the following LP relaxation for Problem~\eqref{problemCP}:
\begin{align}\label{problemMTcl}
\tag{$\mathrm{MT^{clique}}$}
{\rm max}   \quad & \gamma\nonumber\\
 {\rm s.t.} \quad  & X_{ii} -2 X_{ij} + X_{jj} + Y_{ii} -2 Y_{ij} + Y_{jj}\geq \gamma, \; 1 \leq i < j \leq n\label{aux3}\\
       & (x, X) \in \C_x, \; (y, Y) \in \C_y\nonumber \\
       &   X_{ii} \leq x_i, \; Y_{ii} \leq y_i, \; i \in [n]\label{aux4}\\
      & x\in [0,1]^n, \; y \in [0,1]^n \nonumber.
\end{align}
The polyhedron $\C_x$ contains exponentially many clique inequalities, and separating over these inequalities is NP-hard. Yet as we prove shortly,~\eqref{problemMTcl} gives weak upper bounds for the circle packing problem.

%

\vspace{-0.1in}
\paragraph{Tighter variable bounds.}
Using the tighter bounds on $x$ and $y$ given by~\eqref{bnds}, we obtain the following multi-row LP relaxation of Problem~\eqref{problemCP}:
\begin{align}\label{problemMTbnd}
\tag{$\mathrm{MTbnd}^{\rm clique}$}
{\rm max}   \quad & \gamma\nonumber\\
{\rm s.t.} \quad  & X_{ii} -2 X_{ij} + X_{jj} + Y_{ii} -2 Y_{ij} + Y_{jj}\geq \gamma, \; 1 \leq i < j \leq n\nonumber\\
       & (x, X) \in \C^{\rm bnd}_x, \quad (y, Y) \in \C^{\rm bnd}_y\nonumber \\
       &   X_{ii} \leq \frac{x_i}{2}, \; i \in [n_x], \quad X_{ii} \leq x_i, \; n_x+1 \leq i \leq n \label{rev6}\\
       &   Y_{ii} \leq \frac{y_i}{2}, \; i \in [n_y], \quad  Y_{ii} \leq y_i, \; n_y+1 \leq i \leq n \label{rev7}\\
    & x \in [0,1]^n, \; y \in [0,1]^n, \nonumber
\end{align}
where $\C^{\rm bnd}_x(x,X)$ can be obtained from $\C_x(\hat x, \hat X)$ defined before, via the one-to-one linear mapping:
\begin{eqnarray}\label{mapping}
\left.
\begin{array}{ll}
x_i = \frac{\hat x_i}{2},  \;  \forall \; 1\leq i\leq n_x, \quad
x_i = \hat x_i,   \; \forall \; n_x+1 \leq i\leq n \\
X_{ij} = \frac{\hat X_{ij}}{4}, \;      \forall \;  1 \leq i < j \leq n_x, \quad
X_{ij} = \frac{\hat X_{ij}}{2}, \;     \forall \;  1 \leq i \leq n_x < j \leq n, \quad
X_{ij} = \hat X_{ij},  \;      \forall \;  n_x < i < j \leq n.
\end{array}
\right.
\end{eqnarray}
The polyhedron $\C^{\rm bnd}_y$ can be constructed in a similar manner.
\vspace{-0.1in}
\paragraph{Order constraints.}
Next, we obtain a multi-row LP relaxation of Problem~\eqref{problemCP}
by incorporating the order constraints~\eqref{order}. To this end, consider the following reformulation of Problem~\eqref{problemCP}:
\begin{align*}\label{problemCPrord}
\tag{$\mathrm{CPr}_{\rm ord}$}
{\rm max}   \quad & \gamma\\
{\rm s.t.} \quad   &(x, y, \beta, \gamma) \in \P_{\rm ord} \cap \K,
\end{align*}
where $\K$ is the convex cone defined by~\eqref{rev2}, and
\begin{equation}\label{rev4}
\P_{\rm ord} := \big\{(x,y, \beta, \gamma): \;  (x_j-x_i)^2 + (y_j - y_i)^2 \geq \beta_{ij}, \; 1 \leq i < j \leq n, \; 0 \leq x_1 \leq \ldots \leq x_n \leq 1, \; y \in [0, 1]^n \big\}.
\end{equation}
We then convexify Problem~\eqref{problemCPrord} by replacing $\P_{\rm ord} \cap \K$ with the convex set $\conv(\P_{\rm ord}) \cap \K$.
The following lemma provides an extended formulation for the convex hull of $\P_{\rm ord}$.
\begin{lemma}\label{lem2}
    Consider the set $\P_{\rm ord}$ defined by~\eqref{rev4}. Then an extended formulation for $\conv(\P_{\rm ord})$
    is given by:
    \begin{align*}
    &  x_j - x_i + Y_{ii} -2 Y_{ij} + Y_{jj}\geq \beta_{ij}, \; 1 \leq i < j \leq n\\
    & (y,Y) \in \S_y, \; Y_{ii} \leq y_i, \; i \in [n], \; 0 \leq x_1 \leq \ldots \leq x_n \leq 1.
\end{align*}
\end{lemma}
\begin{proof}
To characterize $\conv(\P_{\rm ord})$, it suffices to characterize the convex hull of the set:
$$
\Q_{{\rm ord}} := \Big\{(x, \zeta): (x_j-x_i)^2 \geq \zeta_{ij}, \; 1 \leq i < j \leq n, \; 0 \leq x_1 \leq \ldots \leq x_n \leq 1\Big\},
$$
as $\conv(\P_{\rm ord})$ is the image of $\conv(\Q_{{\rm ord}})\times \conv (\Q_y)$ under a linear mapping,
where $\times$ denotes the Cartesian product, and as before we define $\Q_y = \{(y, Y): y^2_i \geq Y_{ii},\; \forall 1 \leq i \leq  n, \; Y_{ij} \geq y_i y_j, \; \forall 1 \leq i < j \leq n, \; y\in [0,1]^n\}$. From the proof of Lemma~\ref{lem1} we have
$\conv(\Q_y) = \S_y \cap \{(y,Y): Y_{ii} \leq y_i, \forall i \in [n]\}$.
We claim that
\begin{equation}\label{convOrd}
\conv(\Q_{{\rm ord}}) = \Big\{(x, \zeta): x_j-x_i \geq \zeta_{ij}, \; 1 \leq i < j \leq n, \; 0 \leq x_1 \leq \ldots \leq x_n \leq 1\Big\}.
\end{equation}
To see this, first note that the polyhedron defined by~\eqref{convOrd} is a valid relaxation of $\Q_{{\rm ord}}$ as it is obtained by replacing the set $\Q^{i,j}_{\rm ord}:=\{(x_i, x_j, \zeta_{ij}): \;(x_j-x_i)^2 \geq \zeta_{ij},\; 0 \leq x_i \leq x_j \leq 1 \}$ with its convex hull
$\conv(\Q^{i,j}_{\rm ord}) =\{(x_i, x_j, \zeta_{ij}): \; x_j-x_i \geq \zeta_{ij}, 0 \leq x_i \leq x_j \leq 1 \}$
for all $1 \leq i < j \leq n$.
 In addition, this relaxation coincides with $\conv(\Q_{{\rm ord}})$, as
the projection of each point in  $\conv(\Q^{i,j}_{\rm ord})$ onto  the $(x_i, x_j)$
space can be uniquely determined as a convex combination of the vertices of the simplex $0 \leq x_i \leq x_j \leq 1$. 
\end{proof}

Hence, by replacing $\P_{\rm ord} \cap \K$ with $\conv(\P_{\rm ord}) \cap \K$ in Problem~\eqref{problemCPrord}, and using Lemma~\ref{lem2}, we obtain the following LP relaxation of Problem~\eqref{problemCP}:
\begin{align*}\label{problemMTord}\tag{MTord}
{\rm max}   \quad & \gamma\\
{\rm s.t.} \quad  & x_j-x_i + Y_{ii} -2 Y_{ij} + Y_{jj}\geq \gamma, \; 1 \leq i < j \leq n\\
       &  (y, Y) \in \C_y, \quad Y_{ii} \leq y_i, \; \forall i \in [n]\\
       &  0 \leq x_1 \leq \ldots \leq x_n \leq 1.
\end{align*}
Comparing Problem~\eqref{problemTWord} and Problem~\eqref{problemMTord}, we conclude that in the $x$ space, the multi-row relaxations coincide with the single-row relaxations and hence do not lead to any improvements in the relaxation quality.
We are not able to solve Problem~\eqref{problemMTord} analytically; hence we next present a relaxation of this problem that we can solve analytically. More importantly, unlike Problem~\eqref{problemMTord}, this LP relaxation can be solved efficiently in practice. Recall that clique inequalities~\eqref{cliqueP} are facet-defining for $\C_y$; denote by $\T_y$ the polyhedron defined by all clique inequalities with $m=3$ and $\alpha=1$. Notice that $\T_y$ contains $\binom{n}{3}$ inequalities. Let us now replace $\C_y$ by $\T_y$ to obtain the following relaxation of Problem~\eqref{problemMTord}:
\begin{align}\label{problemMTordtri}
\tag{$\mathrm{MTord^{tri}}$}
{\rm max}   \quad & \gamma\\
 {\rm s.t.} \quad  & x_j-x_i + Y_{ii} -2 Y_{ij} + Y_{jj}\geq \gamma, \quad 1 \leq i < j \leq n\label{aux5}\\
       &  y_i + y_j + y_k - Y_{ij} - Y_{jk} - Y_{ik} \leq 1, \quad 1 \leq i < j < k \leq n\label{aux6}\\
       &  Y_{ii} \leq y_i, \quad \forall i \in [n]\label{aux7}\\
       &  0 \leq x_1 \leq \ldots \leq x_n \leq 1, \; y \in [0,1]^n.\nonumber
\end{align}

\paragraph{Best multi-row LP relaxations.} We consider the multi-row LP relaxations of the circle packing problem by combining the
order constraints on $x$ variables~\eqref{semiord}, and tighter bounds on $x$ and $y$ variables~\eqref{bnds}.
As we did in the construction of Problem~(TWcomb), we first
replace each quadratic term $(x_j-x_i)^2$, $1\leq i < j \leq n$ by its concave envelope over the corresponding rectangular, triangular or trapezoidal domain.
Moreover, we let $(y, Y) \in \T^{\rm bnd}_y(y,Y)$,
where $\T^{\rm bnd}_y(y,Y)$ can be obtained from $\T_y(\hat y,\hat Y)$ defined above, via the one-to-one linear mapping:
\begin{eqnarray}\label{mapping2}
\left.
\begin{array}{ll}
y_i = \frac{\hat y_i}{2},  \;  \forall \; 1\leq i\leq n_y, \quad
y_i = \hat y_i,   \; \forall \; n_y+1 \leq i\leq n \\
Y_{ij} = \frac{\hat Y_{ij}}{4}, \;      \forall \;  1 \leq i < j \leq n_y, \quad
Y_{ij} = \frac{\hat Y_{ij}}{2}, \;     \forall \;  1 \leq i \leq n_y < j \leq n, \quad
Y_{ij} = \hat Y_{ij},  \;      \forall \;  n_y < i < j \leq n.
\end{array}
\right.
\end{eqnarray}
Finally we impose inequalities~\eqref{rev7} on $(y,Y)$, and refer to the resulting LP relaxation as Problem~$\rm (MTcomb^{tri})$.

\medskip

The following proposition provides optimal values of the proposed multi-row LP relaxations.

\begin{proposition}\label{th2}
Consider the multi-row LP relaxations of the circle packing problem defined above:
\begin{itemize}
\item [(i)] Let $n \geq 3$. Then the optimal value of Problem~\eqref{problemMTcl} is
$\gamma^* = 1+\frac{1}{n}$ if $n$ is odd, and is
$\gamma^* = 1+\frac{1}{n-1}$ if $n$ is even.

\item [(ii)] Let $n \geq 5$ and let $n_y = \lceil n/4\rceil$. Then the optimal value of Problem~\eqref{problemMTbnd} is
$\gamma^* = \frac{1}{4}(1+\frac{1}{n_y})$ if $n_y$ is odd, and is $\gamma^* = \frac{1}{4}(1+\frac{1}{n_y-1})$ if $n_y$ is even.

\item [(iii)] Let $n \geq 3$. Then the optimal value of Problem~\eqref{problemMTordtri} is
$\gamma^* = \frac{2}{3} (1+\frac{1}{\lfloor(n-1)/2\rfloor})$.

\item [(iv)] Let $n \geq 5$ and let $n_y = \lceil n/4\rceil$. Then the optimal value of Problem~$(\rm MTcomb^{tri})$ is
$\gamma^* = \frac{1}{2}$ for $n \leq 8$ and is
$\gamma^* = \frac{1}{6}(1+\frac{1}{\lfloor(n_y-1)/2\rfloor})$ for $n \geq 9$.
\end{itemize}
\end{proposition}

\begin{proof}
To find the optimal value of each LP, we first find an upper bound on its objective function value by considering a specific subset of constraints and subsequently show that the upper bound is sharp by providing a feasible point that attains the same objective value.

\vspace{0.1in}
\noindent
    \textbf{Part~$(i)$.}
   Summing up all inequalities~\eqref{aux3} and~\eqref{aux4} we obtain:
   \begin{equation}\label{objMT}
\gamma \leq  \frac{4}{n(n-1)} \Big(\frac{(n-1)}{2} \sum_{i=1}^n{x_i} - \sum_{1 \leq i < j \leq n} { X_{ij}}+\frac{(n-1)}{2} \sum_{i=1}^n{y_i} - \sum_{1 \leq i < j \leq n} { Y_{ij}}\Big).
\end{equation}
Recall that inequalities~\eqref{cliqueP} are facet defining for $\C_x$. First, suppose that $n \geq 3$ is odd, implying that $\frac{n-1}{2}$ is an integer and $\frac{n-1}{2} \leq n-2$.
Letting $\alpha = \frac{n-1}{2}$, in~\eqref{cliqueP} yields
\begin{equation}\label{clique}
\frac{(n-1)}{2} \sum_{i \in I} {x_i}- \sum_{i, j \in I, i < j} {X_{ij}} \leq \frac{(n-1)(n+1)}{8}.
\end{equation}
From~\eqref{objMT} and~\eqref{clique}, we deduce that if $n$ is odd, an upper bound
on the optimal value of Problem~\eqref{problemMTcl} is given by $\tilde \gamma = 1+ \frac{1}{n}$ for all $n \geq 3$.
Next, let $n \geq 4$ be even, and let $I$ denote a subset of $[n]$ of cardinality $n-1$. By summing up inequalities~\eqref{aux3}
over all $i < j \in I$, and inequalities~\eqref{aux4} over all $i \in I$ and following a similar line of arguments as above, we deduce that for an even $n \geq 4$, an upper bound on on the optimal value of Problem~\eqref{problemMTcl} is given by $\tilde \gamma = 1+ \frac{1}{n-1}$.
Next we construct a feasible solution of Problem~\eqref{problemMTcl} whose objective value equals $\tilde \gamma$, implying $\gamma^* = \tilde \gamma$.
Suppose that $n$ is odd and consider the point
\begin{equation}\label{point}
\tilde x_i = \tilde X_{ii} = \tilde y_i = \tilde Y_{ii} = \frac{1}{2}\big(1+\frac{1}{n}\big), \; \forall \; i \in [n], \quad \tilde X_{ij} = \tilde Y_{ij}= \frac{1}{4}\big(1+\frac{1}{n}\big), \; \forall \; 1 \leq i < j \leq n.
\end{equation}
First note that $\tilde x_i -2 \tilde X_{ij} + \tilde x_j =\tilde y_i -2 \tilde Y_{ij} + \tilde y_j = \frac{\tilde\gamma}{2}$ for all $1 \leq i < j \leq n$.
Thus, to prove feasibility of~\eqref{point}, we need to show that $(\tilde x, \tilde X) \in \C_x$
(or equivalently $(\tilde y, \tilde Y) \in \C_y$).
To do so, it suffices to prove that the optimal value of the
following bivariate integer quadratic program is zero:
\begin{eqnarray*}
&{\rm max}   \; & -\frac{1}{8}\big(1+\frac{1}{n}\big) m^2 + \frac{1}{2} \big(1+\frac{1}{n}\big) \alpha m - \frac{1}{2} \alpha^2
+\frac{1}{8} \big(1+\frac{1}{n}\big) m -\frac{\alpha}{2}\\
& {\rm s.t.}   & 1 \leq \alpha \leq m-2, \quad  3 \leq m \leq n \\
&      & \alpha, \; m, \; {\rm integer}.
\end{eqnarray*}
It can be checked that the Hessian of the objective function is indefinite inside the feasible region, while the restriction of the quadratic function to each edge
of the domain is concave. By examining the concave univariate function over each edge, we find that the maximum
is either attained along the edge $m = n$ at $\alpha = \frac{n+1}{2}$ or $\alpha = \frac{n-1}{2}$, or is attained along the edge $\alpha = 1$ at $m = 3$
and is equal to zero. For an even $n$, we can use a similar line of arguments by considering the point $\tilde x_i = \tilde X_{ii} = \tilde y_i = \tilde Y_{ii} =\frac{1}{2}(1+\frac{1}{n-1})$ for $i \in [n]$,
and $\tilde X_{ij} = \tilde Y_{ij} = \frac{1}{4}(1+\frac{1}{n-1})$ for $1 \leq i < j \leq n$.
Thus the optimal value Problem~~\eqref{problemMTcl} is
$\gamma^* = \tilde \gamma$.

\vspace{0.1in}
\noindent
\textbf{Part~$(ii)$.} Suppose $n \geq 5$, so that $n_y \geq 2$. Consider the following constraints of Problem~\eqref{problemMTbnd}:
\begin{align*}
& X_{ii} - 2 X_{ij} + X_{jj} + Y_{ii} - 2 Y_{ij} + Y_{jj} \geq \gamma, \; 1 \leq i < j \leq n_y \\
& (x, X)_{n_y} \in \bar \C_x, \; (y, Y)_{n_y} \in \bar C_y \\
&X_{ii} \leq \frac{x_i}{2}, \; Y_{ii} \leq  \frac{y_i}{2}, \; i \in [n_y],
\end{align*}
where $(x, X)_{n_y}$ (resp. $(y, Y)_{n_y}$) consists of the components $x_i$ (resp. $y_i$) with $i \leq n_y$ and $X_{ij}$
(resp. $Y_{ij}$) with $1 \leq i < j \leq n_y$, and
$\bar C_x$ (resp. $\bar C_y$)is defined by all inequalities of $\C^{\rm bnd}_x$ (resp. $\C^{\rm bnd}_x$) containing only variables $(x,X)_{n_y}$ (resp. $(y,Y)_{n_y}$). Then using a similar line of arguments as in Part~$(i)$, we conclude that
for an odd $n$ (resp. even $n$),  an upper bound on the optimal value of Problem~\eqref{problemMTbnd} is
$\tilde\gamma = \frac{1}{4}(1+\frac{1}{n_y})$ (resp. $\tilde\gamma = \frac{1}{4}(1+\frac{1}{n_y-1})$).
Next we present a feasible point of Problem~\eqref{problemMTbnd} whose objective value is equal to $\tilde\gamma$, implying $\gamma^* = \tilde \gamma$.
Consider the point
\begin{eqnarray}\label{point2}
\begin{array}{ll}
\tilde x_i = \tilde y_i = \frac{1}{4}(1+\frac{1}{n_y}), \; i \in [n_y],\quad \tilde x_i = \hat y_i = \frac{1}{2}, \; i \in \{n_y+1, \ldots, n\}, \\
\tilde X_{ii} = \frac{\tilde x_i}{2}, \; i \in [n_x], \quad
\tilde X_{ii} = \tilde x_i, \; i \in \{n_x+1 , \ldots, n\},\\
\tilde Y_{ii} = \frac{\tilde y_i}{2}, \; i \in [n_y],
\quad \tilde Y_{ii} = \tilde y_i, \; i \in \{n_y+1 , \ldots, n\},\\
\tilde X_{ij} = \tilde Y_{ij} = \frac{1}{16} (1+\frac{1}{n_y}), \; 1 \leq i < j \leq n_y,\quad
\tilde X_{ij} = \tilde Y_{ij} = \frac{1}{8} (1+\frac{1}{n_y}), \; 1 \leq i \leq n_y < j \leq n,\\
\tilde X_{ij} = \tilde Y_{ij} = \frac{1}{4},  \; n_y < i  < j \leq n.
\end{array}
\end{eqnarray}
It can be checked that $\tilde X_{ii} - 2 \tilde X_{ij} + \tilde X_{jj} + \tilde Y_{ii} - 2 \tilde Y_{ij} + \tilde Y_{jj} = \tilde \gamma$
for all $1 \leq i < j \leq n_y$,
while $\tilde X_{ii} - 2 \tilde X_{ij} + \tilde X_{jj} + \tilde Y_{ii} - 2 \tilde Y_{ij} + \tilde Y_{jj} > \tilde \gamma$, otherwise.
In addition, we have $\tilde X_{ij} = \tilde x_i \tilde x_j$ and $\tilde Y_{ij} = \tilde y_i \tilde y_j$ for all $1 \leq i < j \leq n$ with $j > n_y$, implying that all inequalities in $\C^{\rm bnd}_x \setminus \bar \C_x$ and $\C^{\rm bnd}_y \setminus \bar C_y$ are satisfied.
Hence, to prove feasibility, it suffices to show that
$(\tilde x, \tilde X)_{n_y} \in \bar \C_x$ (or equivalently $(\tilde y, \tilde Y)_{n_y} \in \bar \C_y$).
As the point defined by~\eqref{point} belongs to $\C_x$, and $\bar \C_x$ and $(\tilde x, \tilde X)_{n_y}$ are the image of
$\C_x$ and point~\eqref{point} under the same linear mapping, respectively, it follows that $(\tilde x, \tilde X)_{n_y} \in \bar \C_x$. Hence $\gamma^* = \tilde \gamma$.

\vspace{0.1in}
\noindent
\textbf{Part~$(iii)$.} For each $l \in [n-2]$, summing up three of inequalities~\eqref{aux5} with $(i,j) \in \{(l,l+1), (l,l+2), (l+1, l+2)\}$, multiplying the resulting inequality by $\frac{1}{2}$, and then adding this inequality to inequalities~\eqref{aux6} with $(i,j,k)=(l,l+1, l+2)$, and three of inequalities~\eqref{aux7} with $i \in \{l,l+1, l+2\}$ we get
$\gamma \leq \frac{2}{3} (x_{l+2}-x_l +1)$. Hence,
the optimal value of the following problem is an upper bound on the optimal value of Problem~\eqref{problemMTord}:
\begin{eqnarray*}
&{\rm max}   \quad & \gamma\nonumber\\
& {\rm s.t.}   & \frac{2}{3} (x_{i+2}-x_i +1) \geq \gamma, \; i \in [n-2]\\
& & 0 \leq x_1 \leq \ldots \leq x_n \leq 1.
\end{eqnarray*}
Define $\Delta x:= \frac{1}{\lfloor(n-1)/2\rfloor}$. The optimal value of the above problem is attained at  $\tilde x_i = \lfloor\frac{i-1}{2}\rfloor \Delta x$ for all $i \in [n]$ and $\tilde \gamma = \frac{2}{3} (1+\Delta x)$.
We now give a feasible point of Problem~\eqref{problemMTordtri} whose objective value equals $\tilde \gamma$, implying $\gamma^* = \tilde \gamma$. Consider the point
\begin{equation}\label{pp}
\tilde x_i = \Big\lfloor\frac{i-1}{2} \Big\rfloor\Delta x, \; i \in [n], \quad \tilde y_i = \tilde Y_{ii} = \frac{1+\Delta x}{3}, \; i \in [n], \quad
\tilde Y_{ij} = \frac{j-i}{4}\Delta x,  \; 1 \leq i  < j \leq n,
\end{equation}
It can be checked that $\tilde x_j - \tilde x_i+ \tilde Y_{ii} - 2 \tilde Y_{ij} + \tilde Y_{jj} = \tilde \gamma$ for all $1 \leq i < j \leq n$.
Moreover, we have $\tilde y_i + \tilde y_j + \tilde y_k - \tilde Y_{ij} - \tilde Y_{jk} - \tilde Y_{ik} = 1 + (1-\frac{k-i}{2})\Delta x \leq 1$, where the inequality is valid
since $k - i \geq 2$. Hence, the point defined by~\eqref{pp} is feasible for Problem~\eqref{problemMTordtri}, implying that its optimal value is given by $\gamma^*= \tilde \gamma$.

\vspace{0.1in}
\noindent
\textbf{Part~$(iv)$.} Suppose $n \geq 9$,
so that $n_y \geq 3$. Consider the following constraints of Problem~$\rm (MTcomb^{tri})$:
\begin{eqnarray}
&&\frac{x_j-x_i}{2} + Y_{ii} - 2 Y_{ij} + Y_{jj} \geq \gamma, \quad 1 \leq i < j \leq n_y,\label{rr1}\\
&& \frac{y_i}{2} + \frac{y_j}{2} + \frac{y_k}{2} - Y_{ij} -  Y_{ik} - Y_{jk} \leq \frac{1}{4}, \quad 1 \leq i < j < k \leq n_y\label{rr2}\\
&&Y_{ii} \leq \frac{y_i}{2}, \quad i \in [n_y]\label{rr3}\\
&&0 \leq x_1 \leq \ldots \leq x_{n_y} \leq \frac{1}{2}.\nonumber
\end{eqnarray}
For each $l \in [n_y-2]$, summing three of inequalities~\eqref{rr1} with $(i,j) \in \{(l,l+1), (l,l+2),(l+1,l+2)\}$ and three of inequalities~\eqref{rr3} with $i \in \{l,l+1, l+2\}$, multiplying the resulting inequality by $\frac{1}{2}$ and then adding it to inequality~\eqref{rr2} with $(i,j,k)=(l,l+1,l+2)$, we get $\gamma \leq \frac{1}{3}(x_{l+2}-x_l+\frac{1}{2})$. Using a similar line of arguments as in Part~$(iii)$, we deduce that an upper bound on the optimal value of Problem~$\rm (MTcomb^{tri})$ is given by
$\tilde \gamma = \frac{1}{6}(1+\frac{1}{\lfloor(n_y-1)/2\rfloor})$. We now present a feasible point of Problem~$\rm (MTcomb^{tri})$ whose objective value is $\tilde \gamma$ implying that $\gamma^* = \tilde \gamma$. Define $\Delta x := \frac{1}{2\lfloor(n_y-1)/2\rfloor}$. Consider the point:
\begin{eqnarray}\label{ppp}
\begin{array}{ll}
\tilde x_i = \lfloor\frac{i-1}{2}\rfloor \Delta x, \; i \in [n_y], \quad \tilde y_i = \frac{1+2 \Delta x}{6}, \; i \in [n_y], \quad \tilde x_i = \tilde y_i = \frac{1}{2}, \; i = n_y+1, \ldots, n\\
\tilde Y_{ii} = \frac{\tilde y_i}{2}, \; i \in [n_y], \quad
\tilde Y_{ii} = \tilde y_i, \; i = n_y+1, \ldots, n\\
\tilde Y_{ij} = \frac{j-i}{8}\Delta x,  \; 1 \leq i  < j \leq n_y, \quad \tilde Y_{ij} = \frac{1+ 2 \Delta x}{12}, \; 1 \leq i \leq n_y < j \leq n, \quad\tilde Y_{ij} = \frac{1}{4}, \; n_y < i < j \leq n.
\end{array}
\end{eqnarray}
First note that $\tilde x$ satisfy the order and bounds constraints on $x$ variables and $\tilde y$ satisfy the bound constraints on $y$ variables. Inequalities~\eqref{rev7} are also clearly satisfied.
Moreover, it can be checked that all
inequalities~\eqref{rr1} are satisfied tightly. The remaining inequalities of the form~\eqref{rr1} for $j > n_y$ can be written as
\begin{equation}\label{remain}
\ell_{\D}(x_i, x_j)+Y_{ii}-2Y_{ij}+Y_{jj} \geq \gamma, \quad \forall 1\leq i < j \leq n \; {\rm such \; that} \; j > n_y,
\end{equation}
where $\ell_{\D}(x_i, x_j)$ denotes a facet of the concave envelope of $(x_j-x_i)^2$ over domain $\D$.
Since $0 \leq (x_j-x_i)^2 \leq \ell_{\D}(x_i, x_j)$ for all $x_i,x_j \in \D$, to show that point~\eqref{ppp} satisfies inequalities~\eqref{remain}, it suffices to show that
$Y_{ii}-2Y_{ij}+Y_{jj} \geq \tilde \gamma$ for all $1\leq i < j \leq$  with $j > n_y$. Two cases arise:
\begin{itemize}[leftmargin=*]
\item $n_y < i < j \leq n$: in this case we have
$\tilde Y_{ii}-2 \tilde Y_{ij}+\tilde Y_{jj} = \frac{1}{2}-2(\frac{1}{4})+\frac{1}{2}=\frac{1}{2} \geq \tilde \gamma$, where the inequality follows since $n_y \geq 3$, implying that $\tilde \gamma \leq \frac{1}{3}$.

\item $1 \leq i \leq n_y < j \leq n$:
in this case we have $\tilde Y_{ii}-2 \tilde Y_{ij}+\tilde Y_{jj} = \frac{1+2\Delta x}{12}-2(\frac{1+2\Delta x}{12})+\frac{1}{2} \geq \frac{1}{6}(1+2\Delta x) = \tilde \gamma$, where the inequality follows since $n_y \geq 3$, implying that $\Delta x \leq \frac{1}{2}$.
\end{itemize}
Hence, it remains to show that $(\tilde y, \tilde Y) \in \T^{\rm bnd}_y(y,Y)$. Observe that $\tilde Y_{ij} = \tilde y_i \tilde y_j$ for $1 \leq i < j \leq n$ with $j > n_y$. Therefore, it suffices to consider the following two cases:
\begin{itemize}[leftmargin=*]
    \item $1\leq i < j < k \leq n_y$: in this case we need to show that $(\tilde y, \tilde Y)$  satisfies inequalities~\eqref{rr2}; substituting yields: $\frac{1+2\Delta x}{4}-\frac{(k-i)}{4}\Delta x = \frac{1}{4}(1+(2-k+i)\Delta x)\leq \frac{1}{4}$, where the last inequality follows since $k-i\geq 2$.
    \item $1 \leq i < j \leq n_y < k \leq n$: in this case we need to show that $(\tilde y, \tilde Y)$  satisfies
    $\frac{y_i}{2} + \frac{y_j}{2} + \frac{y_k}{4} - Y_{ij} -  \frac{Y_{ik}}{2} - \frac{Y_{jk}}{2} \leq \frac{1}{4}$.
    substituting yields:$\frac{1+2\Delta x}{6}+\frac{1}{8}-\frac{(j-i)}{8}\Delta x -\frac{1+2\Delta x}{12}= \frac{5}{24}+(\frac{1}{6}-\frac{j-i}{8})\Delta x\leq \frac{1}{4}$, where the inequality follows since $j-i \geq 1$ and $\Delta x \leq \frac{1}{2}$ since $n_y \geq 3$.
\end{itemize}
Hence, point~\eqref{ppp} is feasible for Problem~$\rm (MTcomb^{tri})$ with objective value equal to $\tilde \gamma$, implying $\gamma^* = \tilde \gamma$.
\end{proof}

\section{SDP relaxations}
\label{sec: sdp}

\paragraph{The basic approach.} SDP relaxations are among the most popular convex relaxations for nonconvex quadratically constrained quadratic programs~\cite{vb96}.
The basic idea is to lift the problem to a higher dimensional space by introducing new variables of the form
$X_{ij} = x_i x_j$ (resp. $Y_{ij} = y_i y_j$), for all $1\leq  i \leq j \leq n$, and
subsequently replace the nonconvex set $\{(x, X): X = x x^T\}$ (resp. $\{(y, Y): Y = y y^T\}$) by its convex hull
$\{(x, X): X \succeq xx^T\}$ (resp. $\{(y, Y): Y\succeq yy^T\}$).
It then follows that the following SDP provides an upper bound on the optimal value of Problem~\eqref{problemCP}:
\begin{align}\label{problemSDP1}\tag{SDP1}
{\rm max}   \quad & \gamma\nonumber\\
{\rm s.t.} \quad  & X_{ii}-2 X_{ij}+X_{jj}+Y_{ii}-2Y_{ij}+Y_{jj} \geq \gamma, \quad 1 \leq i < j \leq n\nonumber\\
       &  X \succeq xx^T, \; Y\succeq yy^T\nonumber\\
       &  X_{ii} \leq x_i, \; Y_{ii} \leq y_i,  \quad i \in [n]\label{diagcons}\\
      &  x \in [0,1]^n, \; y \in [0,1]^n.\nonumber
\end{align}

\vspace{-0.1in}
\paragraph{Tighter variable bounds}
As we will prove in Proposition~\ref{th3}, there exists an optimal solution of Problem~\eqref{problemSDP1} that satisfies the tighter bounds
on $x$ and $y$ given by~\eqref{bnds}.
Thus, the simple addition of these constraints to~\eqref{problemSDP1} does not change the optimal value.
However, inequalities~\eqref{bnds} can be utilized to strengthen~\eqref{problemSDP1} by replacing inequalities~\eqref{diagcons} with inequalities~\eqref{rev6}
and~\eqref{rev7}. We refer to the resulting SDP relaxation as Problem~(SDP2).

\vspace{-0.1in}
\paragraph{Order constraints.}
Consider the RLT-type inequalities~\eqref{RLTord} obtained by utilizing order constraints~\eqref{order}.
Let us denote by Problem~(SDP1ord), the SDP relaxation obtained by adding inequalities~\eqref{RLTord}
to~\eqref{problemSDP1}. We now show that Problem~(SDP1ord) has a more compact formulation.
Consider the set
\begin{equation*}
\S := \big\{ (x, X, \zeta): X_{ii}-2 X_{ij} + X_{jj} \geq \zeta_{ij}, \; X_{ii} \leq X_{ij}, \; x_i - X_{ij} \leq x_j - X_{jj}, \; 1 \leq i < j \leq n, X_{ii} \leq x_i, \; i\in [n] \big\}.
\end{equation*}
It can be checked that the projection of $\S$ onto $(x, \zeta)$ is given by $\{(x, \zeta): x_j - x_i \geq \zeta_{ij}, \; 1 \leq i < j \leq n, \; 0 \leq x_1 \leq \ldots \leq x_n \leq 1\}$,
and by proof of Lemma~\ref{lem2}, this set is the convex hull of $\{(x, \zeta): (x_j - x_i)^2 \geq \zeta_{ij}, \; 1 \leq i < j \leq n, \; 0 \leq x_1 \leq \ldots \leq x_n \leq 1\}$.
This in turn implies that the constraint $X \succeq x x^T$ in the description of Problem~(SDP1ord) is redundant and this problem can be equivalently written as:
\begin{align*}\label{problemSDPord}\tag{SDPord}
{\rm max}   \quad & \gamma\\
 {\rm s.t.} \quad  & x_j-x_i+Y_{ii}-2Y_{ij}+Y_{jj} \geq \gamma, \quad 1 \leq i < j \leq n\\
       &  Y\succeq yy^T, \; Y_{ii} \leq y_i,  \quad i\in [n]\\
       &  0 \leq x_1 \leq \ldots \leq x_n \leq 1, \quad y \in [0,1]^n.
\end{align*}

\vspace{-0.2in}
\paragraph{Best SDP relaxations.} We combine the two types of symmetry-breaking constraints to obtain the best SDP relaxations. Define $\I := \{(i,j): 1 \leq i < j \leq n_y\} \cup \{(i,j): n_y + 1 \leq i < j \leq n\}$.
Using inequalities~\eqref{semiord},  for each $(i,j) \in \I$, we generate the following RLT-type inequalities:
$$
X_{ii} \leq X_{ij}, \quad
 \frac{x_i}{2}-X_{ij} \leq \frac{x_j}{2} - X_{jj},  \;{\rm if} \; j \leq n_x,  \quad  x_i-X_{ij} \leq x_j - X_{jj}, \; {\rm if} \; j > n_x.
$$
Adding these inequalities to Problem~(SDP2), we obtain an SDP relaxation denoted by Problem~(SDPcomb).

\medskip
The next proposition provides optimal values of the first two SDPs introduced above.
Anstreicher~\cite{kurt09} conjectured these bounds and verified them
numerically for $3 \leq n \leq 50$ (see Conjectures~4 and~5 in~\cite{kurt09}).

\begin{proposition}\label{th3}
Consider the SDP relaxations of the circle packing problem defined above:
\begin{itemize}
\item [(i)] The optimal value of Problem~\eqref{problemSDP1} is $\gamma^* = 1+\frac{1}{n-1}$ for all $n \geq 2$.

\item [(ii)] The optimal value of Problem~(SDP2) is $\gamma^* =\frac{1}{4}(1+\frac{1}{{\lfloor(n-1)/4\rfloor}})$ for all $n \geq 5$.
\end{itemize}
Furthermore, addition of first-level RLT constraints do not improve the bounds given by (SDP1) and (SDP2).
\end{proposition}
\begin{proof}
\textbf{Part~$(i)$}. Problem~\eqref{problemSDP1} is symmetric in $(x,X)$ and $(y, Y)$. Since the feasible region of this problem is convex, it follows that
there exists an optimal solution with $x = y$ and $X = Y$. Consider $X_{ii} \leq x_i$, $i \in [n]$.
Clearly, at an optimal solution, at least one of these inequalities is binding, since otherwise it is possible to improve the objective value by increasing one of the diagonal entries of $X$.
We argue that at an optimal solution, all these inequalities are binding. To obtain a contradiction, suppose that this is not the case. Denote by $l$ the index of an inequality constraint that is not binding. Since the problem is symmetric in $(x_i, X_{ii})$, $i \in [n]$, it follows that there exists
an optimal solution with $X_{ll} < x_l$ for any
$l \in \{1,\ldots,n\}$. Since the feasible region of this problem is convex, by taking the average over all such solutions, we obtain an optimal solution of (SDP1) for which $X_{ii} < x_i$ for all $i \in [n]$, which is a contradiction. Thus, at an optimal solution $X_{ii} = x_{i}$, for all $i \in [n]$.
Using a similar symmetry argument, we deduce that at an optimal solution $X_{ii} -2 X_{ij} + X_{jj} = \frac{\gamma}{2}$
for all $1 \leq i < j \leq n$. Thus, Problem~\eqref{problemSDP1} simplifies to the following:
\begin{eqnarray}\label{ssdp}
&{\rm max}   \quad & \gamma\\
& {\rm s.t.}   & x_{i}-2 X_{ij}+x_{j} = \frac{\gamma}{2}, \quad 1 \leq i < j \leq n\nonumber\\
&       &  X \succeq xx^T, \quad 0 \leq x \leq 1\nonumber.
\end{eqnarray}
Next, we eliminate $X_{ij}$, $1\leq i <j \leq n$ using the equality constraints.
Define $\bar{X} = \bigl(\begin{smallmatrix}
1&x^T\\ x&\hat{X}
\end{smallmatrix} \bigr)$, where $\hat{X}_{ii} = x_i$, $i\in [n]$ and $\hat{X}_{ij} = \frac{1}{2}(x_i+x_j-\frac{\gamma}{2})$, $1 \leq i < j \leq n$.
Then Problem~(\ref{ssdp}) can be equivalently written as:
\begin{eqnarray}\label{ssdp2}
&{\rm max}   \quad & \gamma\\
& {\rm s.t.}     &  \bar{X} \succeq 0, \quad 0 \leq x \leq 1\nonumber.
\end{eqnarray}
Now consider a feasible solution Problem~\eqref{ssdp2} denoted by $(\tilde x, \tilde \gamma)$. Notice that any permutation of $\tilde x$,
denoted by $\tilde x_{\pi}$ results in a feasible solution of the form $(\tilde x_{\pi}, \tilde \gamma)$.
Since, the feasible set of~(\ref{ssdp2}) is convex, by taking the average of all such feasible points,
we obtain a feasible solution of the form $(\bar x, \tilde \gamma)$, where $\bar x_1 = \bar x_2 = \ldots= \bar x_n$.
Let $\bar x_i = t$ and let $\hat{X}_{ij} = z $. Then Problem~(\ref{ssdp2}) simplifies to a bivariate SDP:
\begin{eqnarray} \label{bisdp1}
&{\rm max}   \; & 4(t-z)\\
& {\rm s.t.}   & A :=
 \left(
\begin{array}{cccc}
t-t^2 & z-t^2 & \ldots &z-t^2\\
z-t^2 & t-t^2 & \ldots  &z-t^2\\
\vdots & \vdots & \ddots &\vdots\\
z-t^2 & z-t^2 &  \ldots& t-t^2
\end{array} \right) \succeq 0\nonumber\\
&       &  0 \leq t \leq 1, \; 0 \leq z \leq 1 \nonumber.
\end{eqnarray}
By direct calculation, it can be shown that the $m$th order principal minor of $A$ is given by:
$$M_m = (t-z)^{m-1} (t+(m-1)z-m t^2), 	\quad 2 \leq m \leq n.$$
Since the objective of~(\ref{bisdp1}) is to maximize $(t-z)$, we can assume that at an optimal solution $t-z > 0$.
Thus, $M_m \geq 0$ if and only if $t+(m-1)z-m t^2 \geq 0$, or equivalently,
$z \geq (m t^2-t)/(m-1)$. In addition, the right-hand side of this inequality is increasing in $m$ for any $t \in [0,1]$.
Thus, for a given $t \in [0,1]$, the matrix $A$ is positive semidefinite if  $z \geq \frac{n t^2-t}{n-1}$ and at the optimal solution
we have $z = \frac{n t^2-t}{n-1}$.
Thus, problem~(\ref{bisdp1}) simplifies to the following univariate optimization problem:
$$\max_{0\leq t \leq 1} \frac{4 n}{n-1} (t-t^2).$$
The optimal value of the above problem is attained at $t = \frac{1}{2}$ and is equal to $\gamma^* = 1+ \frac{1}{n-1}$, which is equal to the optimal value of Problem~\eqref{problemSDP1}. Moreover, an optimal solution of Problem~\eqref{problemSDP1} is attained at:
\begin{equation}\label{pointSDP}
x^*_i = y^*_i = X^*_{ii} = Y^*_{ii} = \frac{1}{2}, \; i\in [n], \quad
X^*_{ij} = Y^*_{ij} =\frac{n-2}{4(n-1)}, \; 1 \leq i < j \leq n.
\end{equation}
Finally, consider the first-level RLT inequalities defined in Remark~\ref{re1}. To show that~\eqref{pointSDP} satisfies these inequalities, it suffices to have $X^*_{ij} \geq x^*_i + x^*_j-1$ for all $1 \leq i < j \leq n$, which is clearly valid since $n \geq 2$.
Therefore the addition of RLT constraints to Problem~\eqref{problemSDP1} does not strengthen the relaxation.

\vspace{0.1in}
\noindent
\textbf{Part~$(ii)$.} Suppose $n \geq 5$ so that $n_y \geq 2$.
Consider a relaxation of Problem~(SDP2) that
only contains constraints of (SDP2) corresponding to the points in lower left quadrant of the unit square:
\begin{align*}\label{problemSDPr}\tag{SDPr}
{\rm max}   \quad & \gamma\\
 {\rm s.t.}   \quad & X_{ii}-2 X_{ij}+X_{jj}+Y_{ii}-2Y_{ij}+Y_{jj} \geq \gamma, \quad 1 \leq i < j \leq n_y\\
       &  X \succeq xx^T, \quad Y\succeq yy^T\\
    &  X_{ii} \leq \frac{x_i}{2}, \quad Y_{ii} \leq \frac{y_i}{2},  \quad i \in [n_y]\\
       &  0 \leq x_i \leq \frac{1}{2}, \quad 0 \leq y_i \leq \frac{1}{2}, \quad  i \in [n_y].
\end{align*}
Note that in the above problem, $X$ and $Y$ are $n_y\times n_y$ matrices.
The optimal value of Problem~\eqref{problemSDPr} is an upper bound on the optimal value of Problem~(SDP2).
The important property of Problem~\eqref{problemSDPr}, however, is its symmetry in $x$ and $y$ variables.
Thus, we can employ a similar line of arguments as in
Part~$(i)$, to deduce that
%
%
%
the optimal value of Problem~\eqref{problemSDPr} is $\tilde \gamma = \frac{1}{4}(1+\frac{1}{n_y-1})$.
We now construct a feasible point of Problem~(SDP2) whose objective equals $\tilde \gamma$, implying that $\gamma^*=\tilde \gamma$. Consider the point:
\begin{eqnarray}\label{point3}
\left.
\begin{array}{ll}
\tilde x_i=\tilde y_i=\frac{1}{4}, \; \forall \; i\in [n_y], \quad
\tilde x_i = \tilde y_i = \frac{1}{2}, \; \forall \; i = n_y+1,\ldots,n, \\
\tilde X_{ii} = \frac{\tilde x_i}{2}, \; \forall \; i\in [n_x], \quad
\tilde X_{ii} = \tilde x_i, \; \forall \; i = n_x+1,\ldots,n, \\
\tilde Y_{ii} = \frac{\tilde y_i}{2}, \; \forall \; i\in [n_y],\quad
\tilde Y_{ii} = \tilde y_i, \; \forall \; i = n_y+1,\ldots,n,\\
\tilde X_{ij} = \tilde Y_{ij} = \frac{n_y-2}{16(n_y-1)}, \;  \forall \; 1 \leq i <j \leq n_y, \\
\tilde X_{ij} = \tilde Y_{ij} = \frac{1}{8}, \; \forall \; 1 \leq i \leq n_y < j \leq n,\quad
\tilde X_{ij}= \tilde Y_{ij} = \frac{1}{4}, \; \forall \; n_y+1 \leq i < j \leq n,
\end{array}
\right.
\end{eqnarray}
It can be checked that $\tilde X_{ii}-2 \tilde X_{ij}+ \tilde X_{jj} + \tilde Y_{ii}-2 \tilde Y_{ij}+\tilde Y_{jj} = \tilde \gamma$
for $1 \leq i < j \leq n_y$, while
$\tilde X_{ii}-2 \tilde X_{ij}+ \tilde X_{jj} + \tilde Y_{ii}-2 \tilde Y_{ij}+\tilde Y_{jj} = \frac{1}{2} \geq \tilde \gamma$ for $1 \leq i < j \leq n_x$ with $j > n_y$, $\tilde X_{ii}-2 \tilde X_{ij}+ \tilde X_{jj} + \tilde Y_{ii}-2 \tilde Y_{ij}+\tilde Y_{jj} = \frac{3}{4} \geq \tilde \gamma$ for $1 \leq i \leq n_x < j \leq n$, and
$\tilde X_{ii}-2 \tilde X_{ij}+ \tilde X_{jj} + \tilde Y_{ii}-2 \tilde Y_{ij}+\tilde Y_{jj} = 1 \geq \tilde \gamma$, for $n_x < i < j \leq n$. Thus, it remains to
show $\tilde X-\tilde x \tilde x^T \succeq 0$ and
$\tilde Y - \tilde y \tilde y^T \succeq 0$.
Let $\bar{X} = \bigl(\begin{smallmatrix}
1&x^T\\ x& X
\end{smallmatrix} \bigr)$ and $\tilde{Y} = \bigl(\begin{smallmatrix}
1&y^T\\ y& Y
\end{smallmatrix} \bigr)$. To show $\bar{X}  \succeq 0$ at $(\tilde x, \tilde X)$ (resp. $\tilde{Y}  \succeq 0$ at $(\tilde y, \tilde Y)$),
it suffices to factorize $\bar{X}$ as $\bar{X}=L D^x L^T$ (resp. $\bar{Y}$ as $\bar{Y}=L' D^y L'^T$),
where $L$ (resp. $L'$) is a lower triangular matrix with ones in the diagonal and
$D^x ={\rm diag}(d^x)$, $d^x \in \R^n$ (resp. $D^y ={\rm diag}(d^y)$, $d^y \in \R^n$) is a nonnegative diagonal matrix.
That is, we need to find $L$ and $d^x$ such that
\begin{equation}\label{factor}
\bar X_{ij} = \sum_{k=1}^i{L_{ik} d^x_k L_{jk}}, \quad \forall 1\leq i \leq j \leq n.
\end{equation}
By direct calculation it can be checked that the following
choices satisfy equation~\eqref{factor}:
\begin{equation*}
d^x_j :=  \left\{
\begin{array}{ll}
 1, \quad & j=1\\
 \frac{1}{16}\prod_{i=1}^{j-2}\big(1-\frac{1}{(n_y-i)^2}\big), \; & 2 \leq j \leq n_y+1\\
 0, &  n_y+1 < j \leq n_x+1\\
 \frac{1}{4}, & {\rm otherwise},
\end{array} \right.
\quad
L_{ij} :=  \left\{
\begin{array}{ll}
 1,  & i = j \\
 \tilde x_{i-1}, \quad & j = 1, 2 \leq i \leq n+1\\
 \frac{-1}{(n_y-j+1)}, \; & 2 \leq j < i \leq n_y+1\\
 0, & {\rm otherwise }.
\end{array} \right.
\end{equation*}
%
%
Similarly, by defining $L' := L$, $d^y_i := d_i^x$ for $i \in [n_y+1]$,
and $d^y_i := \frac{1}{4}$, for $i = n_y+2, \ldots, n+1$, we deduce that $\bar{Y}=L' D^y L'^T$.
Hence, we have shown that~\eqref{point3}
is feasible for Problem~(SDP2) implying that the optimal value of this problem is given by $\gamma^* = \tilde \gamma$.

Finally, consider the first-level RLT constraints of Problem~(SDP2) obtained from those of Problem~(SDP1) under linear mappings~\eqref{mapping} and~\eqref{mapping2} for $(x,X)$ and $(y,Y)$, respectively. We show that~\eqref{point3} satisfies RLT inequalities of $(x,X)$ variables. The proof for RLT inequalities of $(y,Y)$ variables follows from a similar line of arguments. For each $1\leq i < j \leq n_x$, we have $X_{ij} \geq \frac{x_i}{2}+\frac{x_j}{2}-\frac{1}{4}$, substituting~\eqref{point3} gives $\frac{n_y-2}{16(n_y-1)} \geq \frac{1}{8}+\frac{1}{8}-\frac{1}{4}$, which holds since $n_y \geq 2$.  For each $1\leq i \leq n_x < j \leq n$, we have $X_{ij} \geq x_i+\frac{x_j}{2}-\frac{1}{2}$, substituting~\eqref{point3} gives $\frac{1}{8} \geq \frac{1}{4}+\frac{1}{4}-\frac{1}{2}$. For each $n_x < i < j \leq n$, we have $X_{ij} \geq x_i+x_j-1$, substituting~\eqref{point3} gives  $\frac{1}{4} \geq \frac{1}{2}+\frac{1}{2}-1$. Hence, adding first-level RLT
constraints to Problem~(SDP2) does not strengthen the relaxation.
\end{proof}

Since we are not able to solve Problem~(SDPord) and Problem~(SDPcomb) analytically, we perform numerical experiments to compare their strength with the proposed LP relaxations.
In Figure~\ref{f1a}, we compare the optimal value of Problem~(SDPord) with that of Problem~$(\rm MTord^{tri})$, as given by Part~$(iii)$ of Proposition~\ref{th2}, for $3 \leq n \leq 30$. As can be seen from the figure, while the SDP bounds are slightly stronger than the LP counterparts for $n > 13$, the relative gap between the two bounds is below five percent for all $n$.
In Figure~\ref{f1b},
we compare the the optimal value of Problem~(SDPcomb) with that of Problem~$(\rm MTcomb^{tri})$, as given by Part~$(iv)$ of Proposition~\ref{th2}, for $5 \leq n \leq 30$. As can be seen from the figure, while
the bounds given by the two relaxations coincide for $5 \leq n \leq 8$, the multi-row LP bounds are stronger than the SDP bounds for $9 \leq n \leq 30$.

\begin{figure}[htb]
    \centering
    \subfigure[]{\label{f1a}
    \epsfig{figure=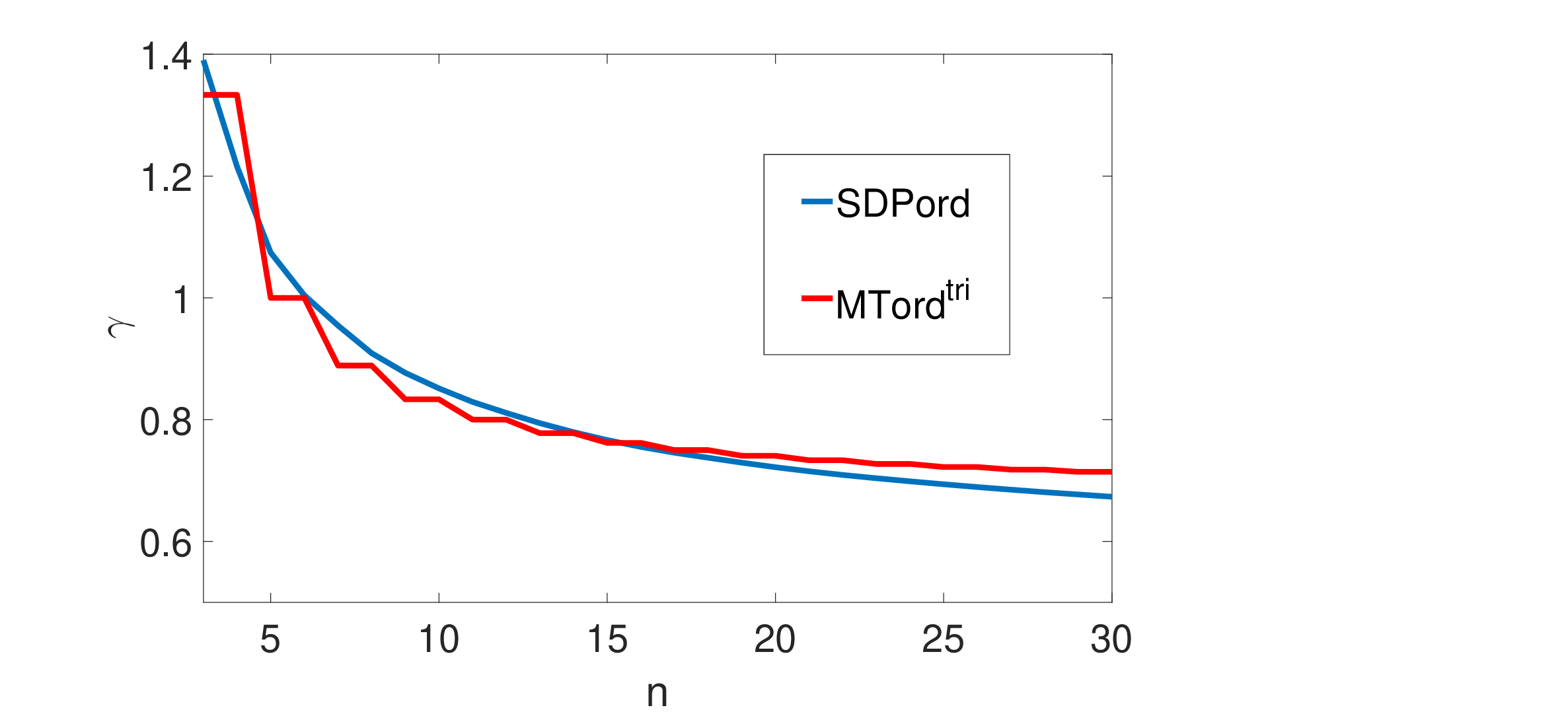, scale=0.25, trim=20mm 0mm 100mm 0mm,clip}}
    \subfigure[]{\label{f1b}
    \epsfig{figure=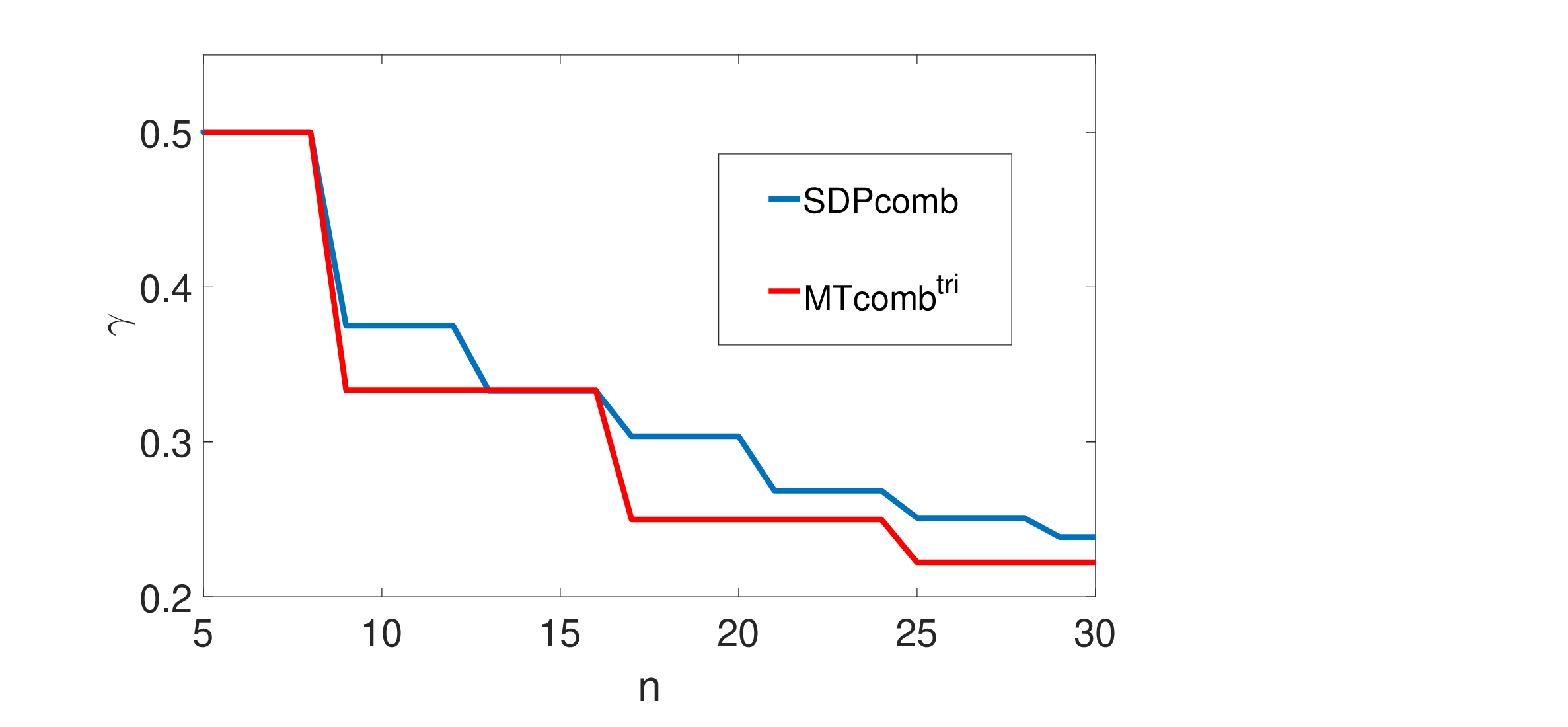, scale=0.25, trim=20mm 0mm 100mm 0mm,clip}}
    \caption{The upper bounds for the circle packing problem obtained by SDP versus LP relaxations. 
    }
    \label{f1}
\end{figure}

\begin{figure}[htb]
    \centering
    \epsfig{figure=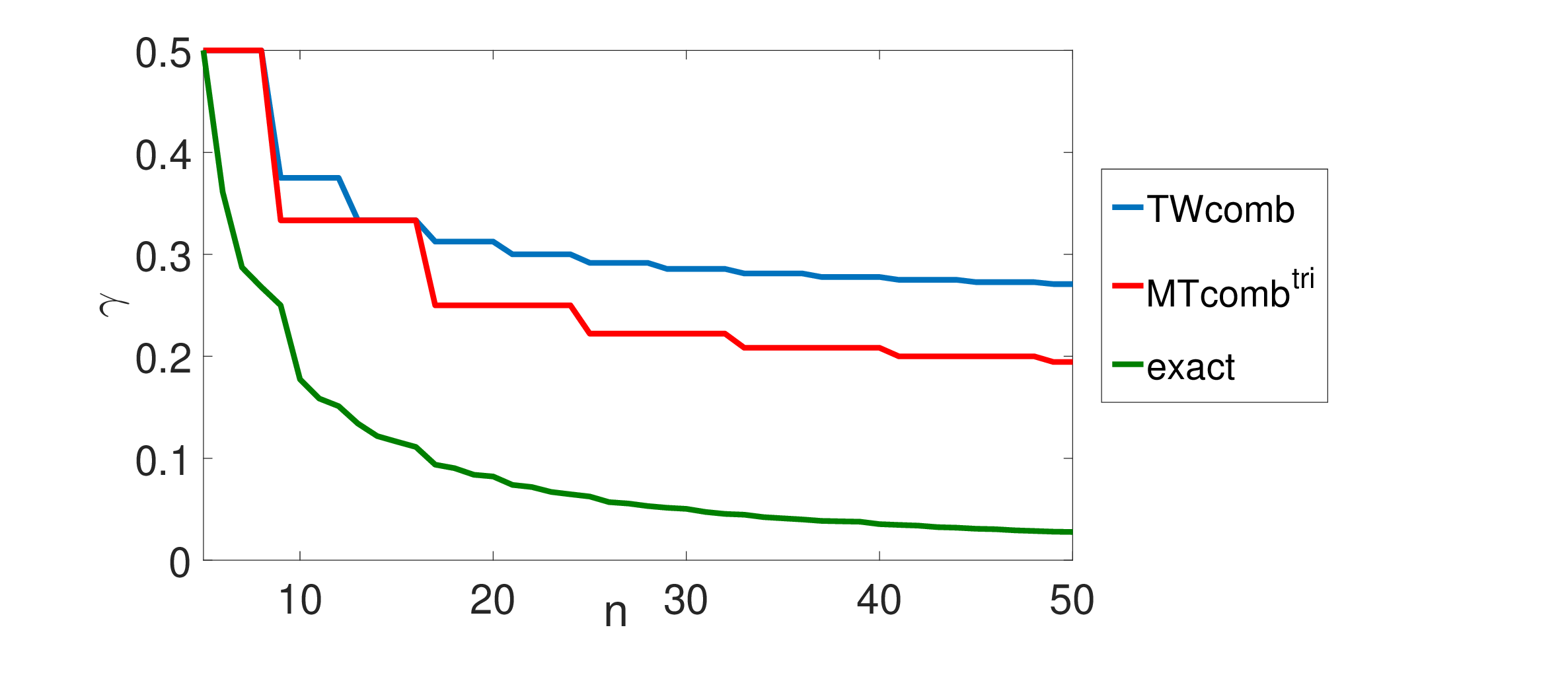, scale=0.25, trim=20mm 0mm 50mm 0mm,clip}
    \caption{The upper bounds obtained by the best single-row LP relaxation (TWcomb)
    and the best multi-row LP relaxation (${\rm MTcomb^{tri}}$) versus the optimal value of the circle packing problem (exact).}
    \label{f3}
\end{figure}

\vspace{-0.2in}
\section{Key takeaways}
\label{sec: conclude}

We conducted a theoretical assessment of several convexification techniques for the circle packing problem. Our main findings are stated in Propositions~\ref{th1}-\ref{th3}: from Propositions~\ref{th1} and~\ref{th3} it follows that $(i)$ the bound given by the basic SDP relaxation (\ie Problem~\eqref{problemSDP1}) is identical to that of the single-row LP relaxation with order constraints (\ie Problem~\eqref{problemTWord}), $(ii)$ the bound given by the SDP relaxation with tightened variable bounds (\ie Problem~(SDP2)) is identical to that of the best single-row LP relaxation (\ie Problem~(MTcomb$^{\rm tri}$)). In addition, via numerical experimentation (see Figure~\ref{f1b}), we observed that the upper bound given by the best multi-row LP relaxation is better than that of the best SDP relaxation. As the computational cost of solving aforementioned LP relaxations is lower than
SDP relaxations, we conclude that for the circle packing problem, LP relaxations are superior to SDP relaxations.

From Propositions~\ref{th1} and~\ref{th2} it follows that the proposed multi-row LP relaxations are considerably better than single-row LP relaxations.
In Figure~\ref{f3}, we plot the optimal values of the best single-row and multi-row LP relaxations along with the optimal value of Problem~(CP)
for $5 \leq n \leq 50$. The exact solutions of the circle packing problem are taken from {\tt www.packomania.com}.
As can be seen from Figure~\ref{f3}, by utilizing certain facets of the BQP, we are able to improve the quality of the upper bound on Problem~(CP) by about $30\%$ for large $n$. However, by increasing $n$, the quality of both upper bounds deteriorates quickly. This observation in turn explains the ineffectiveness of the existing techniques to convexify a nonconvex set defined by collection of non-overlapping constraints.


\begin{footnotesize}
\bibliographystyle{plain}
\bibliography{reference,aida}
\end{footnotesize}

\end{document}